\documentclass[12pt]{amsart}
\usepackage{amssymb}
\usepackage{hyperref}
\usepackage{tabularx}
\usepackage{booktabs}
\usepackage{caption}

\newtheorem{thm}{Theorem}[section]
\newtheorem{lem}[thm]{Lemma}
\newtheorem{cor}[thm]{Corollary}
\newtheorem{prop}[thm]{Proposition}
\newtheorem{ex}[thm]{Example}

\newtheorem*{prob*}{Open problem}

\theoremstyle{definition}
\newtheorem{defi}[thm]{Definition}

\theoremstyle{remark}
\newtheorem{rem}[thm]{Remark}
\newtheorem*{rem*}{Remark}

\DeclareMathOperator{\id}{id}

\DeclareMathOperator{\Aut}{Aut}

\newcommand{\kringel}{\mathbin{\raise1pt\hbox{$\scriptstyle\circ$}}}
\newcommand{\pkt}{\mathbin{\raise0pt\hbox{$\scriptstyle\bullet$}}}

\newcommand{\C}{\mathbb{C}}

\newcommand{\End}{{\rm End}}
\newcommand{\Der}{{\rm Der}}

\newcommand{\La}{\mathfrak{a}}

\newcommand{\Lg}{\mathfrak{g}}
\newcommand{\Lh}{\mathfrak{h}}

\newcommand{\Ll}{\mathfrak{l}}
\newcommand{\Ln}{\mathfrak{n}}

\newcommand{\Lr}{\mathfrak{r}}
\newcommand{\Ls}{\mathfrak{s}}

\newcommand{\im}{\mathop{\rm im}}

\newcommand{\al}{\alpha}
\newcommand{\be}{\beta}
\newcommand{\ga}{\gamma}
\newcommand{\de}{\delta}

\newcommand{\ka}{\kappa}
\newcommand{\la}{\lambda}

\newcommand{\ra}{\rightarrow}

\renewcommand{\phi}{\varphi}

\begin{document}

% Ab hier duerfen Sie wieder.

\title[PA-structures]{Rota--Baxter operators and post-Lie algebra structures on semisimple Lie algebras}
%  Die Kurzfassung kommt oben ueber die Seiten, sie steht in eckigen Klammern
%  Auch Autorennamen koennen eine Kurzfassung haben

\author[D. Burde]{Dietrich Burde}
\author[V. Gubarev]{Vsevolod Gubarev}
\address{Fakult\"at f\"ur Mathematik\\
Universit\"at Wien\\
  Oskar-Morgenstern-Platz 1\\
  1090 Wien \\
  Austria}
\email{dietrich.burde@univie.ac.at}
\address{Fakult\"at f\"ur Mathematik\\
Universit\"at Wien\\
  Oskar-Morgenstern-Platz 1\\
  1090 Wien \\
  Austria}
\email{vsevolod.gubarev@univie.ac.at}

\date{\today}


\begin{abstract}
Rota--Baxter operators $R$ of weight $1$ on $\Ln$ are in bijective correspondence to post-Lie algebra structures 
on pairs $(\Lg,\Ln)$, where $\Ln$ is complete. We use such Rota--Baxter operators to study the existence and 
classification of post-Lie algebra structures on pairs of Lie algebras $(\Lg,\Ln)$, where $\Ln$ is semisimple.
We show that for semisimple $\Lg$ and $\Ln$, with $\Lg$ or $\Ln$ simple, the existence of a post-Lie algebra structure 
on such a pair $(\Lg,\Ln)$ implies that $\Lg$ and $\Ln$ are isomorphic, and hence both simple. If $\Ln$ is semisimple,
but $\Lg$ is not, it becomes much harder to classify post-Lie algebra structures on $(\Lg,\Ln)$, or
even to determine the Lie algebras $\Lg$ which can arise. Here only the case $\Ln=\Ls\Ll_2(\C)$ was studied.
In this paper we determine all Lie algebras $\Lg$ such that there exists a post-Lie algebra structure on $(\Lg,\Ln)$ 
with $\Ln=\Ls\Ll_2(\C)\oplus \Ls\Ll_2(\C)$. 
 \end{abstract}

\maketitle

\section{Introduction}
Rota--Baxter operators were introduced by G. Baxter \cite{BAX} in $1960$ as a formal generalization of 
integration by parts for solving an analytic formula in probability theory. Such operators $R\colon A\ra A$
are defined on an algebra $A$ by the identity
\[
R(x)R(y) = R( R(x)y + xR(y) + \lambda xy )
\]
for all $x,y\in A$, where $\la$ is a scalar, called the {\em weight} of $R$.
These operators were then further investigated, by G.-C. Rota \cite{ROT}, Atkinson \cite{ATK},
Cartier \cite{CAR} and others. In the 1980s these operators were studied in integrable systems 
in the context of classical and modified Yang--Baxter equations \cite{SEM,BED}. 
Since the late 1990s, the study of Rota--Baxter operators has made great progress in many areas, 
both in theory and in applications \cite{HOB, BAG, GUO, GUK, GUR, BEG, GUB}. \\
Post-Lie algebras and post-Lie algebra structures also arise in many areas, e.g., in differential geometry and the study 
of geometric structures on Lie groups. Here post-Lie algebras arise 
as a natural common generalization of pre-Lie algebras \cite{HEL,KIM,SEG,BU5,BU19,BU24} and LR-algebras \cite{BU34, BU38}, 
in the context of nil-affine actions of Lie groups, see \cite{BU41}. A detailed account of the differential 
geometric context of post-Lie algebras is also given in \cite{ELM}.
On the other hand, post-Lie algebras have been introduced by 
Vallette \cite{VAL} in connection with the homology of partition posets and the study of Koszul operads. 
They have been studied by several authors in various contexts, e.g., for algebraic operad triples \cite{LOD}, 
in connection with modified Yang--Baxter equations, Rota--Baxter operators, universal enveloping algebras, double 
Lie algebras, $R$-matrices, isospectral flows, Lie-Butcher series and many other topics \cite{BAG, ELM, GUB}. 
There are several results on the existence and classification of post-Lie algebra structures, in particular
on commutative post-Lie algebra structures \cite{BU51,BU52,BU57}.\\
It is well-known \cite{BAG} that Rota--Baxter operators $R$ of weight $1$ on $\Ln$ are in bijective correspondence to 
post-Lie algebra structures on pairs $(\Lg,\Ln)$, where $\Ln$ is complete. In fact, RB-operators always yield
PA-structures.
So it is possible (and desirable) to use results on RB-operators for the existence and classification of 
post-Lie algebra structures. \\[0.2cm]
The paper is organized as follows. In section $2$ we give basic definitions of RB-operators and
PA-structures on pairs of Lie algebras. We summarize several useful results. For a complete Lie algebra $\Ln$ 
there is a bijection between PA-structures on $(\Lg,\Ln)$ and RB-operators of weight $1$ on $\Ln$. The PA-structure is
given by $x\cdot y=\{R(x),y\}$. Here we study the kernels of $R$ and $R+\id$. If $\Lg$ and $\Ln$ are not
isomorphic, then both $R$ and $R+\id$ have a non-trivial kernel. Moreover, if one of $\Lg$ or $\Ln$ is not
solvable, then at least one of $\ker(R)$ and $\ker(R+\id)$ is non-trivial. \\
In section $3$ we complete the classification of PA-structures on pairs of semisimple Lie algebras $(\Lg,\Ln)$,
where either $\Lg$ or $\Ln$ is simple. We already have shown the following in \cite{BU41}. 
If $\Lg$ is simple, and there exists a PA-structure on $(\Lg,\Ln)$, then also $\Ln$ is simple, and we have 
$\Lg\cong \Ln$ with $x\cdot y=0$ or $x\cdot y=[x,y]$. Here we deal now with the case that $\Ln$ is simple.
Again it follows that $\Lg$ and $\Ln$ are isomorphic. The proof via RB-operators uses results of Koszul \cite{KOS} and 
Onishchik \cite{ONI}. We also show a result concerning semisimple decompositions of Lie algebras. Suppose
that $\Lg=\Ls_1+\Ls_2$ is the vector space sum of two semisimple subalgebras of $\Lg$. Then $\Lg$ is semisimple.
As a corollary we show that the existence of a PA-structure on $(\Lg,\Ln)$ for $\Lg$ semisimple and $\Ln$ 
complete implies that $\Ln$ is semisimple. \\
In section $4$ we determine all Lie algebras $\Lg$ which can arise by PA-structures on $(\Lg,\Ln)$ with
$\Ln=\Ls\Ll_2(\C)\oplus \Ls\Ll_2(\C)$. This turns out to be much more complicated than the case $\Ln=\Ls\Ll_2(\C)$,
which we have done in \cite{BU41}. By Theorem $3.3$ of \cite{BU44}, $\Lg$ cannot be solvable unimodular. On the
other hand, the result we obtain shows that there are more restrictions than that.

\section{Preliminaries}

Let $A$ be a nonassociative algebra over a field $K$ in the sense of Schafer \cite{SCH}, with $K$-bilinear 
product $A\times A\ra A$, $(a,b)\mapsto ab$. We will assume that $K$ is an arbitrary field of characteristic 
zero, if not said otherwise. 

\begin{defi}
Let $\la\in K$. A linear operator $R\colon A\ra A$ satisfying the identity
\begin{align}
R(x)R(y) & = R\bigl(R(x)y+xR(y)+\la xy\bigr)  \label{rota}
\end{align}
for all $x,y\in A$ is called a {\em Rota--Baxter operator on $A$ of weight $\la$}, or just {\em RB-operator}.
\end{defi}

Two obvious examples are given by $R=0$ and $R=\la \id$, for an arbitrary nonassociative algebra. 
These are called the {\em trivial} RB-operators.
The following elementary lemma was shown in \cite{GUO}, Proposition $1.1.12$.

\begin{lem}
Let $R$ be an RB-operator on $A$ of weight $\lambda$. Then $-R-\la \id$ is an  RB-operator on $A$ of weight $\lambda$,
and $\la^{-1}R$ is an  RB-operator on $A$ of weight $1$ for all $\la\neq 0$.
\end{lem}

It is also easy to verify the following results.  

\begin{prop}\cite{BEG}\label{2.3}
Let $R$ be an RB-operator on $A$ of weight $\lambda$ and $\psi\in\Aut(A)$. Then $R^{(\psi)}=\psi^{-1}R\psi$ is an
RB-operator on $A$ of weight $\la$.
\end{prop}

\begin{prop}\cite{GUO}
Let $B$ be a countable direct sum of an algebra $A$. Then the operator $R$ defined on $B$ by
\[
R((a_1,a_2,\ldots ,a_n,\ldots ))=(0,a_1,a_1+a_2,a_1+a_2+a_3,\ldots )
\]
is an RB-operator on $B$ of weight $1$.
\end{prop}

\begin{prop}\label{2.5}
Let $B=A\oplus A$ and $\psi\in \Aut(A)$. Then the operator $R$ defined on $B$ by
\begin{align}
R((a_1,a_2)) & =(0,\psi(a_1))
\end{align}
is an RB-operator on $B$ of weight $1$. Furthermore the operator $R$ defined on $B$ by
\begin{align}
R((a_1,a_2)) & =(-a_1,-\psi(a_1))
\end{align}
is an RB-operator on $B$ of weight $1$.
\end{prop}

\begin{proof}
Let $x=(a_1,a_2)$ and $y=(b_1,b_2)$. Then we have
\begin{align*}
R(R(x)y+xR(y)+\la xy) & = R((0,\psi(a_1)b_2+(0,a_2\psi(b_1))+(a_1b_1,a_2b_2)) \\
                      & = (0,\psi(a_1b_1)) \\
                      & = (0,\psi(a_1)\psi(b_1))\\
                      & = R(x)R(y).
\end{align*}
The second claim follows similarly.
\end{proof}

\begin{prop}\cite{HOB}\label{2.6}
Let $A=A_1\oplus A_2$, $R_1$ be an RB-operator of weight $\la$ on $A_1$, $R_2$ be an RB-operator of weight $\la$
on $A_2$. Then the operator $R\colon A\ra A$ defined by $R((a_1,a_2))=(R_1(a_1),R_2(a_2))$ is an RB-operator of weight $\la$
on $A$.
\end{prop}

\begin{prop}\cite{GUO}\label{2.7}
Let $A=A_1\dot +A_2$ be the direct vector space sum of two subalgebras. Then the operator $R$ defined on $A$ by
\begin{align}
R(a_1+a_2) & =-\la a_2
\end{align}
for $a_1\in A_1$ and  $a_2\in A_2$ is an RB-operator on $A$ of weight $\la$.
\end{prop}

We call such an operator {\it split}, with subalgebras $A_1$ and $A_2$. Note that the set of
all split RB-operators on $A$ is in bijective correspondence with all decompositions $A=A_1\dot + A_2$ as a direct sum
of subalgebras. 

\begin{lem}\cite{BEG}
Let $R$ be an RB-operator of nonzero weight $\la$ on an algebra $A$. Then $R$ is split if and only if
$R(R+\la \id)=0$.
\end{lem}

\begin{lem}
Let $A=A_-\dot +A_0\dot +A_+$ be a direct vector space sum of subalgebras of $A$. Suppose that $R$ is an RB-operator
of weight $\la$ on $A_0$, $A_{-}$ is an $(R+\id)(A_0)$-module and $A_+$ is an $R(A_0)$-module. Define an operator $P$
on $A$ by
\begin{align}
P_{\mid A_{-}} & = 0, \; P_{\mid A_0} = R,\; P_{\mid A_+} = -\la \id .
\end{align}
Then $P$ is an RB-operator on $A$ of weight $\la$.
\end{lem}

\begin{defi}
Let $P$ be an RB-operator on $A$ defined as above such that not both $A_{-}$ and $A_+$ are zero.
Then $P$ is called {\it triangular-split}.
\end{defi}

We also recall the definition of post-Lie algebra structures on a pair of Lie algebras $(\Lg,\Ln)$ over $K$, see 
\cite{BU41}.

\begin{defi}\label{pls}
Let $\Lg=(V, [\, ,])$ and $\Ln=(V, \{\, ,\})$ be two Lie brackets on a vector space $V$ over 
$K$. A {\it post-Lie algebra structure}, or {\em PA-structure} on the pair $(\Lg,\Ln)$ is a 
$K$-bilinear product $x\cdot y$ satisfying the identities:
\begin{align}
x\cdot y -y\cdot x & = [x,y]-\{x,y\} \label{post1}\\
[x,y]\cdot z & = x\cdot (y\cdot z) -y\cdot (x\cdot z) \label{post2}\\
x\cdot \{y,z\} & = \{x\cdot y,z\}+\{y,x\cdot z\} \label{post3}
\end{align}
for all $x,y,z \in V$.
\end{defi}

Define by  $L(x)(y)=x\cdot y$ the left multiplication operator of the algebra $A=(V,\cdot)$. 
By \eqref{post3}, all $L(x)$ are derivations of the Lie algebra $(V,\{,\})$. Moreover, by \eqref{post2}, 
the left multiplication
\[
L\colon \Lg\ra \Der(\Ln)\subseteq \End (V),\; x\mapsto L(x)
\]
is a linear representation of $\Lg$. \\
If $\Ln$ is abelian, then a post-Lie algebra structure on $(\Lg,\Ln)$ corresponds to
a {\it pre-Lie algebra structure} on $\Lg$. In other words, if $\{x,y\}=0$ for all $x,y\in V$, then 
the conditions reduce to
\begin{align*}
x\cdot y-y\cdot x & = [x,y], \\
[x,y]\cdot z & = x\cdot (y\cdot z)-y\cdot (x\cdot z),
\end{align*}
i.e., $x\cdot y$ is a {\it pre-Lie algebra structure} on the Lie algebra $\Lg$, see \cite{BU41}. 

\begin{defi}
Let $x\cdot y$ be a PA-structure on $(\Lg,\Ln)$. If there exists a $\phi\in\End(V)$ such that
\[
x\cdot y=\{\phi(x),y\}
\]
for all $x,y\in V$, then $x\cdot y$ is called an {\em inner} PA-structure on $(\Lg,\Ln)$.
\end{defi}

The following result is proved in \cite{BAG}, Corollary $5.6$.

\begin{prop}\label{2.13}
Let $(\Ln,\{,\},R)$ be a Lie algebra together with a Rota--Baxter operator $R$ of weight $1$, i.e., a linear operator
satisfying
\[
\{R(x),R(y)\}=R(\{R(x),y\}+\{x,R(y)\}+\{x,y\})
\]
for all $x,y\in V$. Then 
\[
x\cdot y=\{R(x),y\}
\]
defines an inner PA-structure on $(\Lg,\Ln)$, where the Lie bracket of $\Lg$ is given by
\begin{align}\label{9}
[x,y] & =\{R(x),y\}-\{R(y),x\}+\{x,y\}.
\end{align}
\end{prop}

Note that $\ker(R)$ is a subalgebra of $\Ln$. For $x,y\in\ker(R)$ we have $R(\{x,y\})=0$.
% Neu
Recall that a Lie algebra is called {\em complete}, if it has trivial center and only inner derivations.

\begin{prop}
Let $\Ln$ be a Lie algebra with trivial center. Then any inner PA-structure on $(\Lg,\Ln)$ arises by a Rota--Baxter
operator of weight $1$. Furthermore, if $\Ln$ is complete, then every PA-structure on $(\Lg,\Ln)$ is inner.
\end{prop}

\begin{proof}
The first claim follows from Proposition $2.10$ in \cite{BU41}. 
By Lemma $2.9$ in \cite{BU41} every PA-structure on $(\Lg,\Ln)$ with complete Lie algebra $\Ln$ is inner. 
The result can also be derived from the proof of Theorem $5.10$ in \cite{BAG}.
\end{proof}

\begin{cor}\label{2.15}
Let $\Ln$ be a complete Lie algebra. Then there is bijection between PA-structures on $(\Lg,\Ln)$ and RB-operators
of weight $1$ on $\Ln$.
\end{cor}

As we have seen, any inner PA-structure on $(\Lg,\Ln)$ with $Z(\Ln)=0$ arises by a Rota--Baxter
operator of weight $1$. For Lie algebra $\Ln$ with non-trivial center this need not be true.

\begin{ex}
Let $(e_1,e_2,e_3)$ be a basis of $V$ and $\Ln=\Lr_2(K)\oplus K$ with $\{e_1,e_2\}=e_2$. Then
\[
\phi=\begin{pmatrix} 1 & 0 & 0 \cr 0 & -1 & 0 \cr \al & \be & \ga \end{pmatrix}
\] 
defines an inner PA-structure on $(\Lg,\Ln)$ by $x\cdot y=\{\phi(x),y\}$ with $\Lg=\Ln$, i.e., with
$[e_1,e_2]=e_2$. But $\phi$ is not always a Rota--Baxter operator of weight $1$ for $\Ln$.
It is easy to see that this is the case if and only if $\be=0$.
\end{ex}

\begin{prop}
Let $x\cdot y$ be an inner PA-structure arising from an RB-operator $R$ on $\Ln$ of weight $1$. Then $R$ is also
an RB-operator of weight $1$ on $\Lg$, i.e., it satisfies
\begin{align*}
[R(x),R[y)] & = R([R(x),y]+[x,R(y)]+[x,y])
\end{align*}
for all $x,y\in V$.
\end{prop}

\begin{proof}
Because of $R([x,y])=\{R(x),R(y)\}$ and the definition of $[x,y]$ we have
\begin{align*}
R([R(x),y]+[x,R(y)]+[x,y]) & = \{R(R(x)),R(y)\}+\{R(x),R(R(y))\}+\{R(x),R(y)\} \\
                           & =[R(x),R(y)]
\end{align*}
for all $x,y\in V$.
\end{proof}

\begin{cor}
Let $x\cdot y=\{R(x),y\}$ be a PA-structure on $(\Lg,\Ln)$ defined by an RB-operator $R$ of weight $1$ on $\Ln$.
Denote by $\Lg_i$ be the Lie algebra structure on $V$ defined by
\begin{align*}
[x,y]_0 & = \{x,y\},\\
[x,y]_{i+1} & = [R(x),y]_i-[R(y),x]_i+[x,y]_i, 
\end{align*}
for all $i\ge 0$. Then $R$ defines a PA-structure on each pair $(\Lg_{i+1},\Lg_i)$.
\end{cor}

We have $[x,y]_1=[x,y]$, and both $R$ and $R+\id$ are Lie algebra homomorphisms from $\Lg_{i+1}$ to  $\Lg_i$,
see Proposition $7$ in \cite{SEM}. Hence we obtain a composition of homomorphisms
\[
\Lg_i \xrightarrow[R+\id]{R} \Lg_{i-1}\xrightarrow[R+\id]{R} \cdots  \xrightarrow[R+\id]{R} \Lg_{0}
\]
So the kernels $\ker(R^i)$ and $\ker((R+\id)^i)$ are ideals in $\Lg_j$ for all $1\le i\le j$. \\[0.2cm]
For a Lie algebra $\Lg$, denote by $\Lg^{(i)}$ the derived ideals defined by $\Lg^{(1)}=\Lg$ and $\Lg^{(i+1)}=[\Lg^{(i)},\Lg^{(i)}]$
for $i\ge 1$. An immediate consequence of Proposition $\ref{2.13}$ is the following observation.

\begin{prop}
Let $x\cdot y=\{R(x),y\}$ be a PA-structure on $(\Lg,\Ln)$ defined by an RB-operator $R$ of weight $1$ on $\Ln$. Then we have
$\dim \Lg^{(i)}\le \dim \Ln^{(i)}$ for all $i\ge 1$. 
\end{prop}

\begin{cor}
Let $x\cdot y$ be a PA-structure on $(\Lg,\Ln)$, where $\Ln$ is complete. Then we have $\dim \Lg^{(i)}\le \dim \Ln^{(i)}$ 
for all $i\ge 1$. In particular, if $\Ln$ is solvable, so is $\Lg$, and if $\Lg$ is perfect, so is $\Ln$.
\end{cor}

\begin{proof}
By Corollary $\ref{2.15}$ this follows from the proposition. 
\end{proof}

\begin{prop}\label{2.21}
Let $x\cdot y=\{R(x),y\}$ be a PA-structure on $(\Lg,\Ln)$ defined by an RB-operator $R$ of weight $1$ on $\Ln$.
Then the following holds.
\begin{itemize}
\item[$(1)$] If $\Lg$ and $\Ln$ are not isomorphic, then both $R$ and $R+\id$ have a non-trivial kernel.
\item[$(2)$] If either $\Lg$ or $\Ln$ is not solvable, then at least one of the operators $R$ and $R+\id$ has a 
non-trivial kernel.
\end{itemize}
\end{prop}

\begin{proof}
For $(1)$, assume that $\ker(R)=0$. Then $R\colon \Lg\ra \Ln$ is invertible, 
hence an isomorphism. This is a contradiction. The same is true for $R+\id$. For $(2)$ assume that $\ker(R)=\ker(R+\id)=0$.
Then $R$ and $R+\id$ are isomorphisms from $\Lg$ to $\Ln$, and $\Lg\cong \Ln$. Then we can apply a result of Jacobson \cite{JAC}
to the automorphism $\psi:=(R+\id)\circ R^{-1}$ of $\Ln$, because $\Ln$ is not solvable. We obtain a nonzero fixed point 
$x\in \Ln$, so that
\[
0 = \psi(x)-x =(R+\id)R^{-1}(x)-x=R^{-1}(x).
\]
Since $R$ is bijective, $x=0$, a contradiction.
\end{proof}

\begin{cor}\label{2.22}
Let $\Ln$ be a simple Lie algebra and $R$ be an invertible RB-operator of nonzero weight $\la$ on $\Ln$.
Then we have $R=-\la \id$.
\end{cor}

\begin{proof}
By rescaling we may assume that $R$ has weight $1$. We obtain a PA-structure on $(\Lg,\Ln)$ by Proposition $\ref{2.13}$, 
with Lie bracket \eqref{9} on $\Lg$. Since $\Ln$ is not solvable, either
$R$ or $R+\id$ have a nontrivial kernel. But $\ker(R)=0$ by assumption, so that $\ker(R+\id)$ is a nontrivial
ideal of $\Ln$. Hence we have $R+\id=0$.
\end{proof}

\section{PA-structures on pairs of semisimple Lie algebras}

We will assume that all algebras in this section are finite-dimensional. 
Let $x\cdot y$ be a PA-structure on $(\Lg,\Ln)$ over $\C$, where $\Lg$ is simple and $\Ln$ is semisimple. 
Then $\Ln$ is also simple, and both $\Lg$ and $\Ln$ are isomorphic, see Proposition $4.9$ in \cite{BU41}. 
We have a similar result for $\Ln$ simple and $\Lg$ semisimple. However, its proof is more difficult
than the first one.

\begin{thm}
Let $x\cdot y$ be a PA-structure on $(\Lg,\Ln)$ over $\C$, where $\Ln$ is simple and $\Lg$ is semisimple. 
Then $\Lg$ is also simple, and both $\Lg$ and $\Ln$ are isomorphic.
\end{thm}

\begin{proof}
By Corollary $\ref{2.15}$ we have $x\cdot y=\{R(x),y\}$ for an RB-operator $R$ of weight $1$ on $\Ln$. Assume
that $\Lg$ and $\Ln$ are not isomorphic. By Proposition $\ref{2.21}$ $(2)$ both $\ker(R)$ and
$\ker(R+\id)$ are proper nonzero ideals of $\Lg$, with $\ker(R)\cap \ker(R+\id)=0$. So we have
\[
\Lg=\ker(R)\oplus \ker(R+\id)\oplus \Ls
\]
with a semisimple ideal $\Ls$. We have $\Ln=\im(R)+\im(R+\id)$ because of
$x=R(-x)+(R+\id)(x)$ for all $x\in \Ln$, and
\begin{align*}
\im (R) & \cong \Lg/\ker(R)\cong \ker(R+\id)\oplus \Ls, \\
\im (R+\id) & \cong \Lg/\ker(R+\id)\cong \ker(R)\oplus \Ls.
\end{align*}
This yields a semisimple decomposition 
\[
\Ln=(\ker(R+\id)\oplus \Ls) + (\ker(R)\oplus \Ls).
\]
Suppose that $\Ls$ is nonzero. Then both summands are not simple. This is a contradiction to Theorem $4.2$ in 
Onishchik's paper \cite{ONI}, which says that at least one summand in a semisimple decomposition of a simple
Lie algebra must be simple. Hence we obtain 
$\Ls=0$, $\im(R)=\ker(R+\id)$, $\im(R+\id)=\im(R)$ and
\[
\Ln=\im(R)\dot + \im(R+\id).
\]
Then the main result of Koszul's note \cite{KOS} implies that $\Ln=\im(R)\oplus \im(R+\id)$, which is a 
contradiction to the simplicity of $\Ln$. Hence $\Lg$ and $\Ln$ are isomorphic.
\end{proof}

If $\Lg$ is semisimple with only two simple summands, we can prove the same result for any field $K$ of characteristic
zero. 

\begin{prop}
Let $x\cdot y$ be a PA-structure on $(\Lg,\Ln)$, where $\Ln$ is semisimple, and $\Lg=\Ls_1\oplus\Ls_2$ 
is the direct sum of two simple ideals of $\Lg$. Then $\Lg$ and $\Ln$ are isomorphic.
\end{prop}

The proof is the same as before. The only argument where we needed the complex numbers, was the result
of \cite{ONI}, which we do not need here. \\[0.2cm]
Let $\Ln=\Ls_1\oplus \Ls_2$ be a direct sum of two simple isomorphic ideals $\Ls_1$ and $\Ls_2$. 
We would like to find all RB-operators of weight $1$ on $\Ln$ such that $\Lg$ with 
bracket \eqref{9} is isomorphic to $\Ln$.

\begin{prop}
All PA-structures on $(\Lg,\Ln)$ with $\Lg\cong\Ln=\Ls_1\oplus \Ls_2$, where $\Ls_1$ and $\Ls_2$ simple isomorphic ideals 
of $\Ln$, arise by the trivial RB-operators or by one of the following RB-operators $R$ on $\Ln$, and
$\psi\in\Aut(\Ln)$,
\begin{align*}
R((s_1,s_2)) & = (-s_1,-\psi(s_1)), \\
R((s_1,s_2)) & = (0,\psi(s_1)), \\
R((s_1,s_2)) & = (-s_1,0)), 
\end{align*}
up to permuting the factors and application of $\phi(R)=-R-\id$ to these operators.
\end{prop}

\begin{proof}
By Proposition $\ref{2.5}$ and Proposition $\ref{2.7}$ the given operators are RB-operators of weight $1$
on $\Ln$, because $R$ is. By Proposition $\ref{2.21}$ at least one of  $\ker(R)$ and $\ker(R+\id)$ is nonzero.
Suppose first that both $\ker(R)$ and $\ker(R+\id)$ are zero. 
Then we have  $\Lg=\ker(R)\oplus \ker (R+\id)$ 
and $\Ln=\ker(R)\dot + \ker(R+\id)$. It is easy to see that $\ker(R)$ coincides with $\Ls_1$ or $\Ls_2$ by using the 
Theorem of Koszul \cite{KOS}. Applying $\phi$ if necessary, we can assume that $\ker(R)=\Ls_2$. 
Then again by Koszul's result we have $R((s_1,s_2))=(\psi_1(s_1),\psi_2(s_1))$ or $R((s_1,s_2))=(\psi_1(s_1),0))$ for some
$\psi_1,\psi_2\in \Aut(\Ln)$. Since $\im(R)=\ker(R+\id)$ we either have $R((s_1,s_2))=(-s_1,-\psi(s_1))$ or
$R((s_1,s_2))  = (-s_1,0)$. \\
In the second case, one of the kernels is zero. Applying $\phi$ if necessary, we may assume that $\ker(R+\id)=0$
and $\ker(R)=\Ls_1$.
Then $\Lg/\ker(R)$ is a simple Lie algebra, and $-R-\id$ is an invertible RB-operator of weight $1$ on $\Lg/\ker(R)$.
By Corollary $\ref{2.22}$ we obtain $-R-\id=-\id$, hence $R=0$ on $\Lg/\ker(R)$. This implies $R^2=0$ on $\Lg$.
The projections of $\im(R)$ to $\Ls_1$ and $\Ls_2$ are either zero or an isomorphism on one factor. So we have
$R((s,0))=(0,\psi(s))$ or $R((s,0))=(\psi_1(s),\psi_2(s))$ for some automorphisms $\psi,\psi_1,\psi_2$. 
But the second operator does not satisfy $R^2=0$, and hence is impossible. Therefore we are done.
\end{proof}

\begin{prop}\label{3.4}
Let $x\cdot y=\{R(x),y\}$ be a PA-structure on $(\Lg,\Ln)$ defined by an RB-operator $R$ of weight $1$ on $\Ln$.
Let $\Ln_1=\ker (R^n)$, $\Ln_2=\ker(R+\id)^n$, $\Ln_3=\im(R^n)\cap \im((R+\id)^n)$ for $n=\dim (V)$. Then
$\Ln=\Ln_1\dot + \Ln_2 \dot + \Ln_3$ with $\{\Ln_1,\Ln_3\} \subseteq \Ln_1$, $\{\Ln_2,\Ln_3\}\subseteq \Ln_2$, and 
$\Ln_3$ is solvable. 
\end{prop}

\begin{proof}
We first show by induction that $\ker(R^i)$ is a subalgebra of $\Ln$, and that 
\[
\{\ker(R^i),\im((R+\id)^i)\}\subseteq \ker(R^i)
\]
for all $i\ge 1$. The case $i=1$ goes as follows. We already know that $\ker(R)$ is a subalgebra of $\Ln$. 
So we have to show that $\{\ker(R),\im (R+\id)\}\subseteq \ker(R)$. Let $x\in \ker(R)$ and $y\in \Ln$. 
Then by \eqref{post1} we have
\begin{align*}
\{x,(R+\id)(y)\} & = \{x,R(y)\}+\{x,y\} \\
                 & = [x,y]+\{y,R(x)\} \\
                 & = [x,y], 
\end{align*}
which is in $\ker(R)$, since this is an ideal in $\Lg$. 
For the induction step $i\mapsto i+1$ consider the iteration of the Lie bracket \eqref{9}
for all $i\ge 0$, given by
\[
[x,y]_i = [x,y]_{i+1}-[R(x),y]_i-[x,R(y)]_i
\]
for all $i\ge 0$. Then 
\begin{align*}
\{x,y\} & = [x,y]_1-[R(x),y]_0-[x,R(y)]_0 \\
        & = [x,y]_2-[R^2(x),y]_0-2[R(x),y]_0-2[R(x),R(y)]_0-2[x,R(y)]_0-[x,R^2(y)]_0
\end{align*}
and so on. Define a degree of a term $[R^l(x),R^k(y)]_m$ by $l+k+m$, and let $x,y\in \ker(R^{i+1})$.
We can iterate the brackets, until the degree of every summand on the right-hand side will be greater than $3i$, so that
all summands either have a term $R^l(x)$ with $l>i$, or a term $R^k(y)$ with $k>i$, or all summands
lie in $[\ker(R^{i+1}),\ker(R^{i+1})]_{i+1}$. By induction hypothesis, such terms will vanish for $l>i$ or $k>i$,
and since $\ker(R^{i+1})$ is an ideal in $\Lg_{i+1}$, we have $\{x,y\}\in \ker(R^{i+1})$, so that $\ker(R^{i+1})$ 
is a subalgebra of $\Ln$. The induction step for the second claim follows similarly. 

Since the image of a subalgebra under the action of an RB-operator is a subalgebra, $\Ln_1$, $\Ln_2$ and their
intersection $\Ln_3$ are subalgebras of $\Ln$. We want to show that $\Ln=\Ln_1\dot + \Ln_2 \dot + \Ln_3$.
Because of $\ker(R^n)\cap \im(R^n)=0$ we have $\Ln=\ker(R^n)\dot +\im(R^n)$. In the same way we have
 $\Ln=\ker((R+\id)^n)\dot +\im((R+\id)^n)$. We obtain
\[
\im(R^n)\cap \ker((R+\id)^n)\dot + \im(R^n)\cap \im((R+\id)^n)\subseteq \im (R^n).
\]
We claim that $\ker((R+\id)^n)\subseteq \im(R^n)$, so that we have equality above. Indeed, for $x\in\ker((R+\id)^n)$ 
we have by the binomial formula
\[
x+\binom{n}{n-1}R(x)+\cdots +\binom{n}{1}R^{n-1}(x)=-R^n(x)\in \im(R^n).
\]
Applying $R^{n-1}$ we obtain $R^{n-1}(x)\in \im(R^n)$ and
\[
x+nR(x)+\cdots +\binom{n}{2}R^{n-2}(x)\in \im(R^n).
\]
Iterating this we obtain $x\in \im(R^n)$. This yields
\begin{align*}
\Ln & = \ker(R^n)\dot +\im(R^n) \\
    & = \ker(R^n)\dot +\ker((R+\id)^n)\dot +\im(R^n)\cap \im((R+\id)^n)\\
    & = \Ln_1\dot +\Ln_2\dot +\Ln_3.
\end{align*}
On $\Ln_3$ both operators $R$ and $R+\id$ are invertible. By Proposition $\ref{2.21}$ part $(2)$ it follows 
that $\Ln_3$ is solvable.
\end{proof}

\begin{cor}\label{3.5}
The decomposition $\Ln=\Ln_1\dot + \Ln_2 \dot + \Ln_3$ induces a decomposition  $\Lg_i=\Ln_1\dot + \Ln_2 \dot + \Ln_3$
for each $i\ge 1$ with the same properties as in the Proposition. 
The Lie algebras $(\Ln_j,[,]_i)$ and $(\Ln_j,[,]_0)$ are isomorphic for $j=1,2,3$.
\end{cor}

\begin{proof}
Since $R$ and $R+\id$ are RB-operators on all $\Lg_i$, we obtain the same decomposition with the same subalgebras.
Note that $R+\id$ is invertible on $\Ln_1$, $R$ is invertible on $\Ln_2$ and both are invertible on $\Ln_3$.
In order to show that $(\Ln_1,[,]_i$ is isomorphic to $(\Ln_1,[,]_0$, we consider
a chain of isomorphisms 
\[
(\Ln_1,[,]_n) \xrightarrow{R+\id} (\Ln_1,[,]_{n-1}) \xrightarrow{R+\id}  \cdots  \xrightarrow{R+\id} (\Ln_1,[,]_0).
\]
In a similar way we can deal with $\Ln_2$ and $\Ln_3$.
\end{proof}

\noindent
{\em Note that Proposition $3.6$ is not correct. Hence the proof of Proposition $3.7$ and $3.8$ 
is invalid. However, the statement of both results is true and we have given a new proof of it in 
our paper \cite{BU64} on decompositions of algebras and post-associative algebra structures.}

\begin{prop}\label{3.6}
Let $\Lg=\Ls_1+\Ls_2$ be the vector space sum of two complex semisimple subalgebras of $\Lg$. Then $\Lg$ is semisimple.
\end{prop}

\begin{proof}
Suppose that the claim is not true and let $\Lg$ be a counterexample of minimal dimension. Then $\Lg$ contains
a nonzero abelian ideal $\La$. Then we obtain
\[
\Lg/\La = \Ls_1/(\Ls_1\cap\La) + \Ls_2/(\Ls_2\cap \La).
\]
Since $\Ls_1\cap \La$ is an abelian ideal $\Ls_1$, it must be zero, i.e., $\Ls_1\cap \La=0$. In the same way
we have $\Ls_2\cap \La=0$. Hence we obtain a semisimple decomposition of  $\Lg/\La$ with $\dim(\Lg/\La)<\dim(\Lg)$.
If $\Lg/\La$ is semisimple, this is a contradiction to the minimality of the counterexample $\Lg$. Otherwise
we may assume that $\Lg$ has $1$-dimensional solvable radical. Then $\Lg$ is reductive, and by Theorem $3.2$
of \cite{ONI}, there are no semisimple decompositions of a complex reductive non-semisimple Lie algebra. Hence we are
done.
\end{proof}

\begin{prop}
Let $x\cdot y=\{R(x),y\}$ be a PA-structure on $(\Lg,\Ln)$ over $\C$, where $\Ln$ is simple, defined by an RB-operator 
$R$ of weight $1$ on $\Ln$, with associated Lie algebras $\Lg_i$ for $i=1,\ldots ,n=\dim(V)$. 
Assume that $\Lg_0=\Ln$ and $\Lg_n$ are semisimple. Then all $\Lg_i$ are isomorphic to $\Ln$.
\end{prop}

\begin{proof}
Since $\Ln_1$ and $ \Ln_2$ are kernels of homomorphisms, they are ideals in $\Lg_n$. The quotient
$\Lg_n/(\Ln_1+\Ln_2)\cong \Ln_3$ is semisimple and solvable by Proposition $\ref{3.4}$. Hence $\Ln_3=0$,
and we obtain $\Lg_n=\ker(R^n)\oplus \ker((R+\id)^n)$. Because of Corollary $\ref{3.5}$ we have the decomposition
$\Lg_i=\ker(R^n)\dot + \ker((R+\id)^n)$ for all $i<n$, where all Lie algebras $(\ker(R^n),[,]_i)$ are isomorphic, and all
Lie algebras $(\ker((R+\id)^n),[,]_i)$ are isomorphic. By Proposition $\ref{3.6}$ all $\Lg_i$ are semisimple. 
By Koszul's result \cite{KOS}, all $\Lg_i$ are isomorphic.
\end{proof}

\begin{prop}
Suppose that there is a post-Lie algebra structure on $(\Lg,\Ln)$ over $\C$, where $\Lg$ is semisimple and
$\Ln$ is complete. Then $\Ln$ must be semisimple.
\end{prop}

\begin{proof}
By Corollary $\ref{2.15}$ the PA-structure is given by $x\cdot y=\{R(x),y\}$, where $R$ is an RB-operator
of weight $1$ on $\Ln$. If at least one of $\ker(R)$ and $\ker(R+\id)$ is trivial, we obtain $\Lg\cong \Ln$
by Proposition $\ref{2.21}$, part $(1)$. 
Otherwise $\Ln=\im(R)+\im(R+\id)$ is the sum of two nonzero semisimple subalgebras. By Proposition $\ref{3.6}$ $\Ln$
is semisimple.
\end{proof}

\section{PA-structures on $(\Lg,\Ln)$ with $\Ln=\Ls\Ll_2(\C)\times \Ls\Ll_2(\C)$}

In \cite{BU41}, Proposition $4.7$ we have shown that PA-structures with $\Ln=\Ls\Ll_2(\C)$ exist
on $(\Lg,\Ln)$ if and only if $\Lg$ is isomorphic to $\Ls\Ll_2(\C)$, or to one of the solvable non-unimodular
Lie algebras $\Lr_{3,\la}(\C)$ for $\la\in \C\setminus \{-1\}$. In this section we want to show an analogous result
for $\Ln=\Ls\Ll_2(\C)\times \Ls\Ll_2(\C)$. Here we will use RB-operators on $\Ln$ and an explicit classification
by Douglas and Repka \cite{DOR} of all subalgebras of $\Ln$. This classification is up to inner automorphisms, but
we will only need the subalgebras up to isomorphisms. Let us fix a basis $(X_1,Y_1,H_1,X_2,Y_2,H_2)$ of $\Ln$ consisting
of the following $4\times 4$ matrices.
\[
X_1=E_{12},\, Y_1=E_{21},\, H_1=E_{11}-E_{22},\,X_2=E_{34},\, Y_2=E_{43},\, H_2=E_{33}-E_{44}. 
\]
We use the following table.\\[0.2cm]
\begin{center}
Table $1$: Complex $3$-dimensional Lie algebras
\end{center}
\vspace*{0.5cm}
\begin{center}
\begin{tabular}{c|c}
$\Lg$ & Lie brackets \\
\hline
$\C^3$ &  $-$ \\
\hline
$\Ln_3(\C)$ & $[e_1,e_2]=e_3$ \\
\hline
$\Lr_2(\C) \oplus  \C$ & $[e_1,e_2]=e_2$ \\
\hline
$\Lr_3(\C)$ & $ [e_1,e_2]=e_2,\, [e_1,e_3]=e_2+e_3 $ \\
\hline
$\Lr_{3,\la}(\C),\, \la\neq 0$ & $[e_1,e_2]=e_2, \, [e_1,e_3]=\la e_3$ \\
\hline
$\Ls\Ll_2(\C)$ &  $[e_1,e_2]=e_3,\, [e_1,e_3]=-2e_1,\, [e_2,e_3]=2e_2$ \\
\end{tabular}
\end{center}
\vspace*{0.5cm}
Among the family $\Lr_{3,\la}(\C)$, $\la\neq 0$ there are still isomorphisms. In fact, $\Lr_{3,\la}(\C)\cong \Lr_{3,\mu}(\C)$ 
if and only if $\mu=\la^{-1}$ or $\mu=\la$. The list of subalgebras $\Lh$ of $\Ln$ is given as follows.
We first list the solvable subalgebras, then the semisimple ones and the subalgebras with a non-trivial 
Levi decomposition. \\[0.2cm]
\begin{center}
Table $2$: Solvable subalgebras
\end{center}
\vspace*{0.5cm}
\begin{center}
\begin{tabular}{c|c|c}
$\dim (\Lh)$ & Representative & Isomorphism type \\
\hline
$1$ &  $\langle X_1\rangle ,\,\langle H_1\rangle,\, \langle X_1+X_2\rangle ,\, \langle X_1+H_2 \rangle ,\,
\langle H_1+aH_2\rangle,\, a\in \C^*$ & $\C$ \\
\hline
$2$ & $\langle X_1,X_2\rangle ,\, \langle X_1,H_2\rangle ,\, \langle H_1,H_2\rangle $ & $\C^2$ \\
\hline
$2$ & $\langle X_1+X_2,H_1+H_2\rangle ,\, \langle X_1,H_1+X_2\rangle ,\, \langle X_1, H_1+a H_2\rangle ,\,a\in \C$ & $\Lr_2(\C)$ \\
\hline
$3$ & $\langle X_1,X_2,H_1+\la H_2\rangle ,\, \la\neq 0 $ & $\Lr_{3,\la}(\C)$, $\la\neq 0$ \\
\hline
$3$ & $\langle X_1,H_1, H_2\rangle ,\, \langle X_1,H_1, X_2\rangle $ & $\Lr_2(\C)\oplus \C$ \\
\hline
$4$ & $\langle X_1,H_1,X_2,H_2\rangle $ & $\Lr_2(\C)\oplus \Lr_2(\C)$ \\
\end{tabular}
\end{center}
\vspace*{0.5cm}
\begin{center}
Table $3$: Semisimple subalgebras and Levi decomposable subalgebras
\end{center}
\vspace*{0.5cm}
\begin{center}
\begin{tabular}{c|c|c}
$\dim (\Lh)$ & Representative & Isomorphism type \\
\hline
$3$ &  $\langle X_1,Y_1,H_1\rangle ,\,\langle X_1+X_2,Y_1+Y_2,H_1+H_2\rangle $ & $\Ls\Ll_2(\C)$ \\
\hline
$4$ &  $\langle X_1,Y_1,H_1,H_2\rangle ,\,\langle X_1,Y_1,H_1,X_2\rangle $ & $\Ls\Ll_2(\C)\oplus \C$ \\
\hline
$5$ &  $\langle X_1,Y_1,H_1,X_2,H_2\rangle $ & $\Ls\Ll_2(\C)\oplus \Lr_2(\C)$ \\
\end{tabular}
\end{center}
\vspace*{0.5cm}
\begin{thm}
Suppose that there exists a post-Lie algebra structure on $(\Lg,\Ln)$, where $\Ln=\Ls\Ll_2(\C)\oplus \Ls\Ll_2(\C)$.
Then $\Lg$ is isomorphic to one of the following Lie algebras, and all these possibilities do occur:
\vspace*{0.2cm}
\begin{itemize}
\item[$(1)$] $\Ls\Ll_2(\C)\oplus \Ls\Ll_2(\C)$.
\item[$(2)$] $\Ls\Ll_2(\C)\oplus \Lr_{3,\la}(\C),\, \la\neq -1$. 
\item[$(3)$] $\Lr_{3,\la}(\C)\oplus \Lr_{3,\mu}(\C),\, (\la,\mu) \neq (-1,-1)$. 
\item[$(4)$] $\Lr_2(\C)\oplus \Lr_2(\C)\oplus \Lr_2(\C)$. 
\item[$(5)$] $\Lr_2(\C)\oplus (\C^3 \ltimes \C)=\langle x_1,\ldots ,x_6\rangle $ and Lie brackets, for
 $\al \neq 0$, $\be \neq 0,-1$
\[
[x_1,x_2]=x_1,\, [x_3,x_6]= x_3,\, [x_4,x_6]=\al x_4,\, [x_5,x_6]=\be x_5.
\] 
\item[$(6)$] $\C\oplus ((\Lr_{3,\la}(\C)\oplus \C)\ltimes \C)=\langle x_1,\ldots ,x_6\rangle $ and Lie brackets, for
$\la\neq 0$, $\al\neq 0,-1$,
\[
[x_2,x_4]=x_2,\, [x_3,x_4]=\la x_3,\, [x_3,x_6]= x_3,\, [x_5,x_6]=\al x_5.
\] 
\item[$(7)$] $(\Lr_{3,\la}(\C)\oplus \C^2)\ltimes \C=\langle x_1,\ldots ,x_6\rangle $ and Lie brackets, for
$\la\neq 0$, $\al_1,\al_2\neq 0$, and $(\la,\al_1,\al_2)\neq (-1,\al_1,-\al_1-1)$,
\[
[x_1,x_3]=x_1,\, [x_2,x_3]=\la x_2,\, [x_2,x_6]= \al_1 x_2,\, [x_4,x_6]=x_4 ,\,[x_5,x_6]=\al_2 x_5.
\] 
\item[$(8)$] $(\C^2\oplus \C^2)\ltimes \C^2=\langle x_1,\ldots ,x_6\rangle $ and Lie brackets
\begin{align*}
[x_1,x_5] & =x_1, \hspace{0.5cm}  [x_2,x_5]=\al_2x_2,\,  [x_3,x_5]=\al_4x_3,\, [x_4,x_5] =\al_6 x_4, \\
[x_1,x_6] & =\al_1 x_1,\, [x_2,x_6] =\al_3x_2,\,  [x_3,x_6]=\al_5x_3,\, [x_4,x_6] =\al_7 x_4, 
\end{align*}
with one of the following conditions:
\begin{align*}
(a) & \quad \al_3=1,\, \al_5=\al_1\al_7,\, \al_6=\al_2\al_4,\; \al_1\al_2\neq 1,\, \al_4,\al_7\neq 0,-1, \\
(b) & \quad \al_4=\al_1-1,\, \al_5=-\al_1,\, \al_6=\al_2(\al_1-1),\, \al_7=\al_1\al_3-\al_1^2\al_2-\al_3,\\ 
\phantom{(b)}  & \quad \al_3-\al_1\al_2\neq 0,\, \al_1\neq 0,1.
\end{align*}
\end{itemize}
\end{thm}

\begin{proof}
By Corollary $\ref{2.15}$ it is enough to consider the RB-operators $R$ of weight $1$ on $\Ln$. Then $\ker(R)$ and
$\ker(R+\id)$ are ideals in $\Lg$. If $R$ is trivial, or one of the kernels is trivial, then we have
$\Lg\cong \Ln$, which is type $(1)$. So we assume that $R$ is non-trivial, both $\ker(R)$ and $\ker(R+\id)$
are non-zero, and $\dim(\ker(R))\ge \dim(\ker(R+\id))$. Then, for $\Ln\not\cong \Lg$, either $\Lg$ has a non-trivial 
Levi decomposition, or $\Lg$ is solvable. \\[0.2cm]
{\it Case 1:} Assume that $\Lg$ has a non-trivial Levi decomposition, i.e., that $\Lg\cong \Ls\Ll_2(\C)\ltimes \Lr$.
We claim that $\Ls\Ll_2(\C)$ is a direct summand of $\Lg$, i.e., $\Lg\cong \Ls\Ll_2(\C)\oplus \Lr$, and that
$\Lr$ is not isomorphic to $\Lr_{3}(\C)$. Then we can argue as follows. Because of Remark $2.12$ of \cite{BU44}, 
$\Lg$ cannot be unimodular, except for $\Lg\cong \Ln$. 
Thus $\Lr$ cannot be unimodular, so that $\Lg$ is isomorphic to $\Ls\Ll_2(\C)\oplus \Lr_{3,\la}(\C)$ with $\la\neq -1$. 
On the other hand, all such algebras do arise by Proposition $\ref{2.6}$ and Proposition $4.7$ of \cite{BU41}. \\[0.2cm]
{\it Case 1a:} Suppose that $\Ls\Ll_2(\C)$ is not contained in $\ker(R)$, $\ker(R+\id)$ as a subalgebra. 
Then $\dim(\ker(R+\id))=1$ and $\dim(\ker(R))\in \{1,2\}$. Let us assume, both have dimension $1$. 
The other case goes similarly. Then we have
$\Lr=\langle x_1,x_2,x_3\rangle$, $\ker(R)=\langle x_1\rangle$ and $\ker(R+\id)=\langle x_2\rangle$. Furthermore
$\im(R)\cong \Ls\Ll_2(\C)\ltimes \langle x_2,x_3\rangle$ and $\im(R+\id)\cong \Ls\Ll_2(\C)\ltimes \langle x_1,x_3\rangle$
are $5$-dimensional subalgebras of $\Ln$. By table $3$,  $\Ls\Ll_2(\C)$ is a direct summand of them. 
This implies that  $\Ls\Ll_2(\C)$ is also a direct summand in $\Lg$. Since both $\ker(R)$ and $\ker(R+\id)$ are ideals
in $\Lr$, we can exclude that $\Lr$ is isomorphic to $\Lr_3(\C)$, and we are done. \\[0.2cm]
{\it Case 1b:} $\Ls\Ll_2(\C)$ is contained in one of $\ker(R)$, $\ker(R+\id)$. Without loss of generality
we may assume that $\Ls\Ll_2(\C)\subseteq \ker(R)$. If $\ker(R)=\Ls\Ll_2(\C)$, then $\Ls\Ll_2(\C)$ is an ideal of $\Lg$, and
we have $\Lg\cong \Ls\Ll_2(\C) \oplus \Lr$, where $\Lr\cong \im(R)\le \Ln$ is not isomorphic to $\Lr_3(\C)$ by table $2$, 
and we are done.
Thus we may assume that $\dim(\ker(R))\ge 4$. If $R$ splits with subalgebras $\ker(R)$ and $\ker(R+\id)$, then 
$\Lg\cong \ker(R)\oplus \ker(R+\id)$, and $\dim (\ker(R))+\dim (\ker(R+\id))=6$.
By table $3$, $\Ls\Ll_2(\C)$ is a direct summand of $\ker(R)$, and hence of
$\Lg$. So we have again $\Lg\cong \Ls\Ll_2(\C) \oplus \Lr$, and $\Lr$ is not isomorphic to $\Lr_3(\C)$.
If $R$ is not split, it remains to consider the case $\dim(\ker(R))=4$ and $\dim(\ker(R+\id))=1$. 
We have $\Lr=\langle x,y,z\rangle$ with $\ker(R)=\Ls\Ll_2(\C)\oplus \langle x\rangle$, $\ker(R+\id)=\langle y\rangle$ 
and $[y,\Ls\Ll_2(\C)]=0$. Assume that  $[z,\Ls\Ll_2(\C)]\neq 0$.
Then $\Ls\Ll_2(\C)$ is not a direct summand of the $5$-dimensional subalgebra $\im(R+\id)$ of $\Ln$, which is a 
contradiction to table $3$. Thus we have $\Lg\cong \Ls\Ll_2(\C)\oplus \Lr$. Since $\Lr$ has two disjoint 
$1$-dimensional ideals $\langle x\rangle$ and $\langle y\rangle$, it is not isomorphic to $\Lr_3(\C)$. \\[0.2cm]
{\it Case 2:} Assume that $\Lg$ is solvable. Then $\im(R)$ and $\im(R+\id)$ are solvable subalgebras of $\Ln$ of dimension
at most $4$ by table $2$. So we have $\dim (\ker(R))\ge \dim (\ker(R+\id))\ge 2$. Thus we have the following four cases:
\begin{align*}
(2a)\quad \dim(\ker(R)) & = 4,\; \dim(\ker(R+\id))=2, \\
(2b)\quad \dim(\ker(R)) & = 3,\; \dim(\ker(R+\id))=3, \\
(2c)\quad \dim(\ker(R)) & = 3,\; \dim(\ker(R+\id))=2, \\
(2d)\quad \dim(\ker(R)) & = 2,\; \dim(\ker(R+\id))=2.
\end{align*}
For the cases $(2a)$ and $(2b)$, $R$ is split since the dimensions add up to $6$. Then $\Lg$ is a direct sum
of two solvable subalgebras, which are both isomorphic to subalgebras of $\Ln$. So we have $\Ln=\ker(R)\dot + \ker(R+\id)$
and $\Lg=\ker(R)\oplus \ker(R+\id)$. \\[0.2cm]
{\it Case 2a:} Since we have only $\Lr_2(\C)\oplus \Lr_2(\C)$ as $4$-dimensional solvable subalgebra of $\Ln$, 
we have $\Lg\cong \Lr_2(\C)\oplus \Lr_2(\C)\oplus \C^2$, which is of type $(3)$ for $(\la,\mu)=(0,0)$, 
or $\Lg\cong \Lr_2(\C)\oplus \Lr_2(\C)\oplus \Lr_2(\C)$, which is of type $(4)$. Both cases can arise.
For the first one we will show this in case $(2b)$. For the second, it follows from Proposition $\ref{2.7}$ with
$\Ln=\langle X_1,H_1,X_2,H_2\rangle \dot + \langle Y_1,Y_2+H_1\rangle$. \\[0.2cm]
{\it Case 2b:} We have $\Lg\cong \Lr_{3,\la}(\C)\oplus \Lr_{3,\mu}(\C)$. The case $(\la,\mu)=(-1,-1)$ cannot arise by
Theorem $3.3$ of \cite{BU41}. The cases $(\la,\mu)=(-1,\mu)$ for $\mu\neq -1$ arise by Proposition $\ref{2.7}$ with
\[
\Ln=\langle X_1,X_2,H_1-H_2\rangle \dot + \langle Y_1,Y_2,H_1+ \mu H_2\rangle.
\]
The other cases with $\la,\mu\neq -1$ arise by Proposition $\ref{2.6}$ and Proposition $4.7$ of \cite{BU41}. \\[0.2cm]
{\it Case 2c:} Here $\Lg$ is isomorphic to $(\Lr_{3,\la}(\C)\oplus \Lr_2(\C))\rtimes \C$ or  
$(\Lr_{3,\la}(\C)\oplus \C^2)\rtimes \C$. In the first case, $\Lr_2(\C)\rtimes \C\cong \im(R)$ is a solvable subalgebra
of $\Ln$, hence isomorphic to $\Lr_{3,\nu}(\C)$ by table $2$. So $\C$ acts trivially on $\Lr_2(\C)$, and 
$\im(R+\id)\cong \Lr_{3,\la}(\C)\rtimes \C\cong \Lr_2(\C)\oplus \Lr_2(\C)$. Then $\Lg\cong \Lr_2(\C)\oplus \Lr_2(\C)\oplus
\Lr_2(\C)$, which we have already considered in Case $(2a)$. 
For $(\Lr_{3,\la}(\C)\oplus \C^2)\rtimes \C$ we need to distinguish $\la=0$ and $\la\neq 0$. \\[0.2cm]
{\it Case 2c, $\la=0$:} By Proposition $\ref{2.3}$ we may assume that $\im(R+\id)=\langle X_1,H_1,X_2,H_2\rangle$.
Since $\ker(R)$ is an ideal of  $\im(R+\id)$ isomorphic to $\Lr_2(\C)\oplus \C$, we have $\ker (R)=\langle X_1,H_1,X_2
\rangle$. Let us consider the characteristic polynomial $\chi_R$ of the linear operator $R$ acting on $\Ln$. 
By assumption on the kernels, $\chi_R(t)=t^3(t+1)^2(t-\rho)$. \\[0.2cm]
{\it Case 2c, $\la=0$, $\rho\neq 0,-1$:} Then $R(x_6)=\rho x_6$ for $x_6=H_2+\al H_1+\be X_1+\ga X_2$.
Since $\ker(R+\id)$ is an abelian $2$-dimensional subalgebra of $\Ln$, we have
\[
\ker(R+\id)=\langle Y_1+\nu_1X_1+\nu_2H_1, Y_2+\nu_3X_2+\nu_4H_2\rangle.
\]
We want to compute $[x,y]$ for $x=x_6$ and $y\in \ker(R+\id)$. By Proposition $\ref{2.13}$ we have, using $R(x_6)=\rho x_6$
\begin{align*}
[x,y] & = \{R(x),y\}-\{R(y),x\}+\{x,y\} \\
      & = \{R(x),y\} \\
      & = \rho\{x,y\}.
\end{align*}
For $x_6=H_2+\al H_1+\be X_1+\ga X_2$ and $y \in \ker(R+\id)$ this yields, using the Lie brackets of $\Ln$ in the
standard basis $\{X_1,Y_1,H_1,X_2,Y_2,H_2\}$,
\begin{align}
[x_6,Y_1+\nu_1X_1+\nu_2H_1] & = \rho((2\al\nu_1-2\be\nu_2)X_1-2\al Y_1+\be H_1), \label{10} \\
[x_6,Y_2+\nu_3X_2+\nu_4H_2] & = \rho((2\nu_3-2\ga\nu_4)X_2-2Y_2+\ga H_2).\label{11}
\end{align}  
Since $\ker(R+\id)$ is an ideal in $\Lg$ and $\rho\neq 0$, both vectors lie again in $\ker(R+\id)$. Comparing
coefficients for the basis vectors we obtain 
\[
\be=-2\al \nu_2,\, \al(\nu_1+\nu_2^2)=0,\, \ga=-2\nu_4, \, \nu_3=-\nu_4^2.
\]
Suppose that $\al=0$. Then $x_6=H_2-2\nu_4X_2$ and $\langle X_1,H_1\rangle\cong \Lr_2(\C)$ is a direct summand of $\Lg$.
Therefore $\Lg\cong \Lr_2(\C)\oplus \C\oplus \Lr_{3,\mu}(\C)$ with $\C=\langle Y_1+\nu_1X_1+\nu_2H_1\rangle$,
$ \Lr_{3,\mu}(\C)=\langle X_2,H_2-2\nu_4X_2,Y_2+\nu_4H_2-\nu_4^2X_2\rangle $, $\mu=-(\rho+1)/\rho$, which we have already 
considered above. Hence we may assume that $\al\neq 0$ and $\nu_1=-\nu_2^2$. Consider a new basis for $\Lg$ (note that
we redefine $x_6$) given by
\begin{align*}
(x_1,\ldots,x_6)& =(X_1,-\frac{1}{2}H_1+\nu_2X_1,X_2,Y_1+\nu_2H_1-\nu_2^2X_1,Y_2+\nu_4H_2-\nu_4^2X_2, \\
                & \hspace{1cm} \frac{1}{2\rho}(H_2+\al H_1-2\al\nu_2X_1-2\nu_4X_2)),
\end{align*}
with Lie brackets
\[
[x_1,x_2]=x_1,\, [x_1,x_6]=-\frac{\rho+1}{\rho}\al x_1,\, [x_3,x_6]=-\frac{\rho+1}{\rho} x_3,\, [x_4,x_6]=\al x_4,
\, [x_5,x_6]=x_6.
\]
This algebra is of type $(5)$, if we replace $x_6$ by $x_6+\frac{\al(\rho+1)}{\rho}x_2$. It arises for the triangular-split
RB-operator $R$ with $A_{-}=\ker(R)=\langle x_1,x_2,x_3\rangle $, $\ker(R+\id)=\langle x_4,x_5\rangle$ and 
$A_0=\langle x_6\rangle$, where $x_6=H_2-2\nu_4X_2$, with the action $R(x_6)=\rho x_6$. \\[0.2cm]
{\it Case 2c, $\la=0$, $\rho= -1$:} We may assume that there exists $x_6=Y_2+v$ such that 
$(R+\id)(x_6)=\mu (H_2+\al H_1+\be X_1+\ga X_2)$ for some non-zero $\mu$ and some $\al,\be,\ga\in \C$. 
Since $\ker(R+\id)$ is an abelian subalgebra we obtain $\al=\be=0$ and  $\ker(R+\id)=\langle H_2+\ga X_2,
Y_1+\nu_1X_1+\nu_2H_1\rangle$. Then we may choose $x_6=Y_2+\ka X_2+\nu_3 H_1+\nu_4 X_1$. Then
\begin{align*}
[x_6,H_2+\ga X_2] & = \{R(x_6),H_2+\ga X_2\} \\
                 & = \{ (R+\id)(x_6)-x_6,H_2+\ga X_2\} \\
                 & = -\{ Y_2+\ka X_2, H_2+\ga X_2\} \\
                 & = -2Y_2+2\ka X_2+\ga H_2.
\end{align*}
This is not contained in $\ker(R+\id)$, which is a contradiction to the fact that $\ker(R+\id)$ is an ideal. \\[0.2cm]
{\it Case 2c, $\la=0$, $\rho= 0$:} Then we have $R(H_2)=\al H_1+\be X_1+\ga X_2\neq 0$ and 
$\ker(R+\id)=\langle Y_1+\nu_1 X_1+\nu_2 H_1, Y_2+\nu_3 X_2+\nu_4 H_2\rangle $. Since $[H_2,Y_2+\nu_1 X_1+\nu_2 H_1]=
\{\ga X_2 , Y_2+\nu_1 X_2+\nu_2 H_2\}$ is in $\ker(R+\id)$, we obtain $\ga=0$. Since $[H_2,Y_1+\nu_1 X_1+\nu_2 H_2]=
\{ \al H_1+\be X_1, Y_1+\nu_1 X_1+\nu_2 H_1\}$  is in $\ker(R+\id)$, we obtain $\al (\nu_1+\nu_2^2)=0$ and 
$\be=-2\al\nu_2$. Since  $R(H_2)\neq 0$ we have $\al\neq 0$, $\nu_1=-\nu_2^2$ and $R(H_2)=\al H_1-2\al \nu_2 X_1$.
Consider a new basis for $\Lg$ given by
\begin{align*}
(x_1,\ldots,x_6)& =(X_1,-\frac{1}{2}H_1+\nu_2X_1,X_2,Y_1+\nu_2H_1-\nu_2^2X_1,Y_2+\nu_3 X_2+\nu_4H_2, -\frac{1}{2}H_2),
\end{align*}
with Lie brackets
\[
[x_1,x_2]=x_1,\, [x_1,x_6]=\al x_1,\, [x_3,x_6]=x_3,\, [x_4,x_6]=-\al x_4.
\]
This algebra is of type $(3)$, if we replace $x_6$ by $x_6-\al x_2$. \\[0.2cm]
{\it Case 2c, $\la\neq 0$:} Then we have $\ker(R)=\langle X_1,X_2,-\frac{1}{2}(H_1+\la H_2)\rangle$. 
We again have $\chi_R(t)=t^3(t+1)^2(t-\rho)$, where we distinguish the cases $\rho\neq 0,-1$, $\rho=-1$ and 
$\rho=0$. \\[0.2cm]
{\it Case 2c, $\la\neq 0$, $\rho\neq 0,-1$:} Then we may assume that $R(x_6)=\rho x_6$ for 
$x_6=H_2+\al H_1+\be X_1+\ga X_2$. As $\ker(R+\id)$ is abelian, we have $\ker(R+\id)=\langle Y_1+\nu_1 X_1
+\nu_2 H_1, Y_2+\nu_3X_2+\nu_4H_2\rangle$. Since $V=\ker(R)\oplus \ker(R+\id)\oplus \langle x_6\rangle$, the
two elements $H_1+\la H_2$ and $H_2+\al H_1$ need to be linearly independent, i.e., $1-\al \la\neq 0$. 
By \eqref{10} and \eqref{11} we obtain $\ga=-2\nu_4$, $\be=-2\al\nu_2$, $\nu_3=-\nu_4^2$ and 
$\al(\nu_1+\nu_2^2)$. Suppose that $\al=0$. 
Then $x_6=H_2-2\nu_4X_2$. Consider a new basis for $\Lg$ given by
\begin{align*}
(x_1,\ldots,x_6)& =(Y_1+\nu_1X_1+\nu_2H_1,X_1,X_2,-\frac{1}{2}(H_1+\la H_2),Y_2-\nu_4^2 X_2+\nu_4H_2, -\frac{1}{2(\rho+1)}H_2),
\end{align*}
with Lie brackets
\[
[x_2,x_4]=x_2,\, [x_3,x_4]=\la x_3,\, [x_3,x_6]=x_3,\, [x_4,x_6]=-\la\nu_4 x_3, \, [x_5,x_6]=-\frac{\rho}{1+\rho}x_5.
\]
This is an algebra of type $(6)$, if we replace $x_4$ by $x_4+\la \nu_4 x_3$. \\
Now we assume that $\al\neq 0$. Consider a new basis for $\Lg$ given by
\begin{align*}
(x_1,\ldots, x_6)& =(X_1,X_2,-\frac{1}{2}(H_1+\la H_2),Y_2-\nu_4^2 X_2+\nu_4H_2, Y_1-\nu_2^2 X_1+\nu_2 H_1, \\
 &  -\frac{1}{2(\rho+1)}(H_2-2\nu_4 X_2+\al (H_1-2\nu_2 X_1))),
\end{align*}
with Lie brackets
\begin{align*}
[x_1,x_3] & = x_1,\, [x_2,x_3]=\la x_2,\, [x_1,x_6]=\al x_1,\, [x_2,x_6]=x_2, \\
[x_3,x_6] & =-\al\nu_2 x_1-\la \nu_4 x_2, \, [x_4,x_6]=\de x_4,\, [x_5,x_6]=\al \de x_5,
\end{align*}
where $\de=-\frac{\rho}{\rho+1}$. 
Replacing $x_6$ by $\frac{1}{\de}(x_6-\al \nu_2 x_1-\nu_4 x_2-\al x_3)$ we obtain
the Lie brackets
\[
[x_1,x_3] = x_1,\,[x_2,x_3]=\la x_2,\, [x_2,x_6]=\al' x_2,\, [x_4,x_6]= x_4,\, [x_5,x_6]=\al x_5,
\]
where 
\[
\al'=\frac{1-\al\la}{\de}=\frac{(\rho+1)(\al\la-1)}{\rho}.
\]
Note that $\al'\neq 0$ and $\al'\neq \al\la -1$ by assumption. In other words, $\al\neq \frac{\al'+1}{\la}$.
Consider a new basis for $\Lg$ given by
\[
(x_1',\ldots ,x_6')=(x_2,x_1,\frac{1}{\la}x_3,x_4,x_5,x_6-\frac{\al'}{\la}x_3),
\]
with Lie brackets
\[
[x_1',x_3'] = x_1',\,[x_2',x_3']=\la' x_2',\, [x_2',x_6']=-\al' \la' x_2',\, [x_4',x_6']= x_4',\, 
[x_5',x_6']=\al x_5',
\]
where $\la'=\frac{1}{\la}$. This is of type $(7)$. Since $\Lr_{3,\la}(\C)\cong \Lr_{3,\la'}(\C)$,
one may check that we do not only have $\al\neq \frac{\al'+1}{\la}$, but also $\al\neq \la-\al'$.
For  $\frac{\al'+1}{\la}\neq  \la-\al'$ we obtain no restriction for $\al$. However, 
for $\frac{\al'+1}{\la}= \la-\al'$ we obtain $\la=-1$ or $\la=\al'+1$, which excludes both 
$(\la,\al',\al)=(-1,\al',-\al'-1)$ and  $(\la,\al',\al)=(\la,\la-1,1)$. Rewriting this in the parameters
of the Lie brackets from type $(7)$, we obtain all cases except for $(\la,\al',\al)=(\la,\la-1,1)$ with $\la\neq -1$. 
These PA-structures
arise by a triangular-split RB-operator with $A_{-}=\ker(R)$, $A_{+}=\ker(R+\id)$ and $A_0=\langle x_6\rangle$
with the action $R(x_6)=\rho x_6$, $\rho\neq 0,-1$.  \\[0.2cm]
{\it Case 2c, $\la\neq 0$,$\rho=-1$:} This leads to a contradiction in the same way as case $2c$ with
$\la=0,\rho=-1$. \\[0.2cm]
{\it Case 2c, $\la\neq 0$,$\rho=0$:} We have $R(H_2)=\al X_1+\be X_2+\ga (H_1+\la H_2)$ and $\ker(R+\id)=
\langle x_4,x_5\rangle$ with $x_4=Y_1+\nu_1 X_1+\nu_2 H_1$, $x_5=Y_2+\nu_3 X_2+\nu_4 H_2$. Similarly to 
$\eqref{10},\eqref{11}$ we obtain $R(H_2)=\ga (H_1-2\nu_2 X_1)+\ga \la (H_2-2\nu_4 X_2)$. This implies that
$\ga\neq 0$ and $x_4=Y_1-\nu_2^2X_1+\nu_2H_1$, $x_5=Y_2-\nu_4^2X_2+\nu_4 H_2$. By setting $x_1=X_1$, $x_2=X_2$,
$x_3=-\frac{1}{2}(H_1+\la H_2)$ and $x_6=\frac{1}{2\ga}H_2$ we obtain a new basis for $\Lg$ with Lie brackets
\begin{align*}
[x_1,x_3] & = x_1,\, [x_2,x_3]=\la x_2,\, [x_1,x_6]=-x_1,\, [x_2,x_6]=\de x_2,\\
[x_3,x_6] & = \nu_2 x_1+\la^2\nu_4 x_2,\, [x_4,x_6]=x_4,\, [x_5,x_6]=\la x_5,
\end{align*} 
where $\de=-\frac{1+\la \ga}{\ga}$ with $\de\neq -\la$. Replacing $x_6$ by $x_6+\nu_2 x_1+\la \nu_4 x_2+x_3$ we obtain
the brackets
\[
[x_1,x_3] = x_1,\,[x_2,x_3]=\la x_2,\,[x_2,x_6]=\al_1 x_2,\, [x_4,x_6]=x_4,\,[x_5,x_6]=\la x_5 
\]
with $\al_1=\de+\la=-\frac{1}{\ga}$. This is of type $(7)$ with $\al_2=\la$. It arises by the triangular-split
RB-operator with $A_{-}=\langle x_1,x_2\rangle$, $A_{+}=\langle x_4,x_5\rangle$ and $A_0=\langle u,v\rangle$,
with $u=\frac{1}{\ga}(H_2-2\nu_4 X_2)$ and $v=H_1-2\nu_2 X_1+\la(H_2-2\nu_4X_2)$, and the action $R(u)=v$, $R(v)=0$. \\[0.2cm]
{\it Case 2d:}  Suppose that one of the kernels $\ker(R)$ and $\ker(R+\id)$ is non-abelian. Without loss of generality, 
let us assume that $\ker(R)\cong \Lr_2(\C)$. Write $\Lg\cong (\ker(R)\oplus \ker(R+\id))\ltimes \langle a,b\rangle$. 
Then $\ker(R)\ltimes \langle a\rangle$ is a $3$-dimensional solvable subalgebra of $\im(R+\id)$. By table $2$ we see 
that it is isomorphic to $\Lr_{2}(\C)\oplus \C$. In this case there exist nonzero 
$a'\in \ker(R)\oplus \langle a \rangle$ and $b'\in \ker(R)\oplus \langle b \rangle$ such that $[a',\ker(R)]=[b',\ker(R)]=0$.
Then $\Lg\cong \ker(R)\oplus (\ker(R+\id)\oplus \langle a',b'\rangle)$ with $\ker(R)\cong \Lr_2(\C)$, and 
$\ker(R+\id)\oplus \langle a',b'\rangle \cong \Lr_2(\C)\oplus \Lr_2(\C)$ by Table $2$. Hence we obtain
$\Lg\cong \Lr_2(\C)\oplus \Lr_2(\C)\oplus \Lr_2(\C)$, which is of type $(4)$. \\
So we may assume that $\ker(R)\cong \ker(R+\id)\cong \C^2$. Then the characteristic polynomial of $R$ has the form
$\chi_R(t)=t^2(t+1)^2(t-\rho_1)(t-\rho_2)$.  \\[0.2cm]
{\it Case 2d, $\rho_1,\rho_2\neq 0,-1$:} Suppose first that either $\rho_1\neq \rho_2$, or that $\rho_1=\rho_2$ and the
eigenspace is $2$-dimensional. Then by Proposition $\ref{3.4}$, $\Ln=\ker(R)\dot + \ker(R+\id)\dot +\langle x_5',x_6' \rangle$
with linearly independent eigenvectors $x_5',x_6'$ corresponding to the eigenvalues $\rho_1$ and $\rho_2$.
Since $\ker(R)$ is an abelian ideal in $\im(R+\id)=\langle X_1,H_1,X_2,H_2\rangle$, we may assume that 
$\ker(R)=\langle X_1,X_2\rangle$ and $[x_5',x_6']=0$. 
The decomposition $\Ln=\ker(R+\id)\dot + \im(R+\id)$ 
shows that $\ker(R+\id)$ has a basis 
$x_3 = Y_1+\al H_1+\nu_3 X_1$, 
$x_4 = Y_2+\be H_2+\nu_4 X_2$. 
Since $[x_5',x_6']=0$, we have 
$x_5' = H_1 + \nu_1 X_1 + \xi_1(H_2 + \nu_2 X_2)$,
$x_6' = H_2 + \nu_2 X_2 + \xi_2(H_1 + \nu_1 X_1)$
with $\xi_1\xi_2\neq1$. So we have by \eqref{10} and \eqref{11} 
$x_3 = Y_1 - \frac{\nu_1}{2}H_1 - \frac{\nu_1^2}{4}X_1$,
$x_4 = Y_2 - \frac{\nu_2}{2}H_2 - \frac{\nu_2^2}{4}X_2$.
Consider a basis for $\Lg$ given by
\[
(x_1,\ldots ,x_6)=(X_1,X_2,x_3,x_4,-\frac{1}{2(1+\rho_1)}x_5',-\frac{1}{2(1+\rho_2)}x_6'),
\]
with Lie brackets
\begin{align*}
[x_1,x_5] & = x_1,\,
[x_1,x_6]=\xi_2 x_1,\, 
[x_2,x_5]=\xi_1 x_2,\, 
[x_2,x_6]= x_2,\\
[x_3,x_5] & =\ga x_3,\, 
[x_3,x_6]=\de \xi_2 x_3,\, 
[x_4,x_5]=\ga \xi_1 x_4,\, 
[x_4,x_6]=\de x_4,
\end{align*}
where 
$\ga=-\frac{\rho_1}{\rho_1+1}$, 
$\de=-\frac{\rho_2}{\rho_2+1}$ 
with $\ga,\de\neq 0,-1$ and $\xi_1\xi_2\neq 1$. This is type $(8a)$. It arises by the triangular-split RB-operator $R$ with 
$A_{-}=\langle X_1,X_2\rangle$, $A_{+}=\langle x_3,x_4\rangle$ and $A_0=\langle x_5,x_6\rangle$, where $R$ 
acts on $A_0$ by $R(x_5)=\rho_1 x_5$ and $R(x_6)=\rho_2 x_6$.
Note that for $\nu_2 = \xi_2 = 0$ and $\xi_1\neq0$
we get type $(7)$ without the restriction 
$(\la,\al_1,\al_2)\neq (\la,\la-1,1)$ for $\lambda\neq-1$,
which we had in Case 2c, $\la\neq 0$, $\rho\neq 0,-1$. \\
Suppose now that $\rho_2=\rho_1\neq 0,-1$, and the eigenspace for $\rho_1$ is $1$-dimensional. 
Let $R(x_5')=\rho_1 x_5'$ and $R(x_6')=x_5'+\rho_1 x_6'$. In the same way as before we have
$x_5' = H_1 + \nu_1 X_1 + \xi(H_2 + \nu_2 X_2)$,
$x_6' = \kappa(H_2 + \nu_2 X_2)$
with $\kappa\neq0$ and
$x_3 = Y_1 - \frac{\nu_1}{2}H_1 - \frac{\nu_1^2}{4}X_1$,
$x_4 = Y_2 - \frac{\nu_2}{2}H_2 - \frac{\nu_2^2}{4}X_2$.
Consider a basis for $\Lg$ given by
\[
(x_1,\ldots ,x_6)=(X_1,X_2,x_3,x_4,-\frac{1}{2(1+\rho_1)}x_5',-\frac{1}{2(1+\rho_1)}x_6'),
\]
with Lie brackets
\begin{align*}
[x_1,x_5] & = x_1,\,
[x_1,x_6]=(\ga+1) x_1,\, 
[x_2,x_5]=\xi x_2,\, 
[x_2,x_6]= (\kappa+\xi+\ga\xi)x_2,\\
[x_3,x_5] & =\ga x_3,\, 
[x_3,x_6]=-(\ga+1) x_3,\, 
[x_4,x_5]=\ga \xi x_4,\, 
[x_4,x_6]=(\kappa\ga-\xi-\ga\xi) x_4,
\end{align*}
where $\ga=-\frac{\rho_1}{\rho_1+1}\neq 0,-1$ and $\kappa\neq0$. This is type $(8b)$. 
It arises by the triangular-split RB-operator $R$ with 
$A_{-}=\langle X_1,X_2\rangle$, $A_{+}=\langle x_3,x_4\rangle$ and $A_0=\langle x_5,x_6\rangle$, where $R$ acts
on $A_0$ by $R(x_5)=\rho_1 x_5$ and $R(x_6)=x_5+\rho_1 x_6$.\\[0.2cm]
{\it Case 2d, $\rho_1=\rho_2=0$:} We have $\Lg=\ker(R+\id)\dot + \im(R+\id)$ and we can assume that 
$\ker(R)=\langle X_1,X_2\rangle $ and $\ker(R+\id)=\langle Y_1+\nu_1 X_1+\nu_2 H_1,Y_2+\nu_3 X_2+\nu_4 H_2\rangle$.
Suppose first that $R(v)=X_1$ and $R(w)=X_2$ for some $v,w$. Then
\[
[Y_1+\nu_1 X_1+\nu_2 H_1,v]=\{Y_1+\nu_1 X_1+\nu_2 H_1,X_1\}
 = -H_1 + 2\nu_2 X_1 \in \ker(R+\id),
\]
which is a contradiction. Otherwise we see from the possible Jordan forms of $R$ that there exist $v,w$ with 
$R(v)=\alpha X_1+\beta X_2\neq 0$ and $R(w)=v$. This leads to a contradiction in the same way. \\[0.2cm]
{\it Case 2d, $\rho_1=0,\rho_2\neq 0,-1$:} This case is analagous to the second part of the case before. \\[0.2cm]
{\it Case 2d, $\rho_1=0,\rho_2=-1$:} As above we may assume that $\im(R+\id)=\langle X_1,X_2,H_1,H_2\rangle$
and $\ker(R)=\langle X_1,X_2\rangle$, and $\al H_1+\be H_2+\ga X_1+\de X_2\in \ker(R+\id)\cap \im (R+\id)$ for some
$\al,\be ,\ga,\de \in \C$. Since $\ker(R+\id)$ is abelian, we may assume that $\ker(R+\id)=\langle H_1+\nu_1 X_1,
Y_2+\nu_2 X_2+\nu_3 H_2\rangle$ for some $\nu_1,\nu_2,\nu_3\in \C$. Let $v\in \ker (R^2)$ such that 
$R(v)=\nu_4 X_1+\nu_5 X_2\neq 0$. Then
\[
[v,Y_2+\nu_2 X_2+\nu_3 H_2]=\{\nu_4 X_1+\nu_5 X_2, Y_2+\nu_2 X_2+\nu_3 H_2\}=\nu_5(H_2-2\nu_3 X_2)\in \ker(R+\id)
\]
implies that $\nu_5=0$. By $[v,H_1+\nu_1X_1]=\{\nu_4 X_1,H_1+\nu_1 X_1\}=-2\nu_4 X_1\in \ker(R+\id)$ we obtain
$\nu_4=0$, which is a contradiction to $R(v)\neq 0$.
\end{proof}

\begin{rem}
The algebras from different types are non-isomorphic, except for algebras of type $(8)$, which have
intersections with type $(3)$ and $(7)$ for certain parameter choices.
\end{rem}

\section*{Acknowledgments}
Dietrich Burde is supported by the Austrian Science Foun\-da\-tion FWF, grant P28079 
and grant I3248. Vsevolod Gubarev acknowledges support by the Austrian Science Foun\-da\-tion FWF, 
grant P28079.

\end{document}